\documentclass{amsart}
\usepackage[bbgreekl]{mathbbol}
\usepackage{amssymb,amsmath,amsthm}
\usepackage[foot]{amsaddr}
\title{Univalent completion}
\author{Benno van den Berg$^1$}
\address{${}^1$ ILLC, Universiteit van Amsterdam, P.O. Box 94242, 1090 GE Amsterdam, the Netherlands. E-mail: bennovdberg@gmail.com.}
\author{Ieke Moerdijk$^2$}
\address{${}^2$ Radboud Universiteit Nijmegen, Institute for Mathematics, Astrophysics, and Particle Physics, Heyendaalseweg 135, 6525 AJ Nijmegen, the Netherlands. E-mail: i.moerdijk@math.ru.nl.}
\date{\today}
\usepackage{ast2}
\usepackage[margin=1.7in]{geometry}

\begin{document}

\maketitle

\begin{abstract}
We review the concept of a univalent fibration and show by elementary means that every Kan fibration in simplicial sets can be embedded in a univalent Kan fibration.
\end{abstract}

\section{Introduction}

A Kan fibration $E \to U$ in the category of simplicial sets is universal in case every other Kan fibration $Y \to X$ satisfying some size restrictions is a (homotopy) pullback. In this context, a typical such size restriction is to require $Y \to X$ to have fibers of cardinality strictly less than a fixed regular cardinal. Such a universal fibration is said to be univalent if the path space of $U$ is equivalent to the space of equivalences between fibers of $E \to U$; see Section 5 below for a precise formulation. This notion of univalence was introduced by Voevodsky in the context of modelling universes in Martin-L\"of style type theories. Subsequently, Voevodsky and others proved that particular constructions of universal fibrations in simplicial sets or in related categories are indeed univalent (see, for example, \cite{cisinski14,gepnerkock15,kapulkinetal14,shulman15b,shulman15}).

Rather than focussing on specific such universal fibrations, the purpose of this note is to show that any fibration can be embedded into a univalent one, the latter being universal whenever the former is.

Although we were partially motivated by work on homotopy theoretic models of type theory, we wish to emphasize that our paper uses only basic definitions and results of the homotopy theory of simplicial sets, and is otherwise entirely self-contained.

\section{Notation and background on groupoids}

In this section we fix some notation and terminology concerning groupoids, actions by groupoids on other objects (i.e., representations), nerves, et cetera.

\subsection{Actions} For a groupoid $\mathbb{G}$, we write
\diag{ {\rm ar}(\mathbb{G}) \ar@<1ex>[r]^{s} \ar@<-1ex>[r]_t & {\rm ob}(\mathbb{G}) }
for the sets of arrows and objects of $\mathbb{G}$, and the source and target maps between them. Composition is denoted $m: {\rm ar}(\mathbb{G}) \times_{{\rm ob}(\mathbb{G})} {\rm ar}(\mathbb{G}) \to {\rm ar}(\mathbb{G})$. A (left) action of $\mathbb{G}$ on a set $X$ consists of maps $\pi: X \to {\rm ob}(\mathbb{G})$ and $\alpha: {\rm ar}(\mathbb{G}) \times_{{\rm ob}(\mathbb{G})} X \to X$ satisfying the usual identities. We usually write $g \cdot x$ for $\alpha(g, x)$ where $g: c \to d$ is an arrow in $\mathbb{G}$ and $x$ is an element of $X$ with $\pi(x) = c$. This action gives rise to a new groupoid denoted $X_\mathbb{G}$, with ${\rm ob}(X_\mathbb{G}) = X$ and ${\rm ar}(X_\mathbb{G}) = {\rm ar}(\mathbb{G}) \times_{{\rm ob}(\mathbb{G})} X$. In fact, the arrows $x \to y$ in $X_\mathbb{G}$ are the arrows $g: \pi x \to \pi y$ in $\mathbb{G}$ with $g \cdot x = y$. This groupoid is known as the \emph{action groupoid} for the action of $\mathbb{G}$ on $X$. It comes equipped with a functor (again) denoted $\pi: X_\mathbb{G} \to \mathbb{G}$. Note that for each object $c$ in $\mathbb{G}$, there is a pullback of groupoids
\diag{ X_c \ar[d] \ar[r] & X_\mathbb{G} \ar[d]^\pi \\
1 \ar[r]_c & \mathbb{G} }
where $X_c = \pi^{-1}(c) \subseteq X$ is viewed as a discrete groupoid (having identity arrows only).

\subsection{Nerves} For a groupoid $\mathbb{G}$, we write $N(\mathbb{G})$ for its nerve. It is the simplicial set whose $n$-simplices are strings of composable arrows
\diaglab{nerve}{ c_0 \ar[r]^{g_1} & c_1 \ar[r] & \ldots \ar[r]^{g_n} & c_n }
in $\mathbb{G}$. This definition of the nerve makes sense for any category, not just for groupoids. But for us, it will be convenient to observe that for a groupoid $\mathbb{G}$, the data \refdiag{nerve} can equivalently be represented in the form
\diaglab{nerveagain}{ & & c_k \ar@{.>}[dr] \ar[dl]^{h_1} \ar[dll]_{h_0} \ar[drrr]^{h_n} \\
c_0 & c_1 & \ldots & \hat{c_k} & \ldots & c_n }
where $h_i: c_k \to c_i$ is the appropriate composition of $g_i$'s and their inverses. This representation has the advantage that all the faces $d_i: N(\mathbb{G})_n \to N(\mathbb{G})_{n-1}$ other that $d_k$ are simply given by deleting the arrow $c_k \to c_i$. (This observation is related to the fact that $N(\mathbb{G})$ is in fact a \emph{symmetric} simplicial set, see, for instance, \cite{grandis01}.)

If $\mathbb{G}$ acts on $X$, then $N(X_\mathbb{G})_n = X \times_{{\rm ob}(\mathbb{G})} N(\mathbb{G})_n$. This holds for both presentations \refdiag{nerve} and \refdiag{nerveagain}. Thus, an $n$-simplex in $N(X_\mathbb{G})$ can be represented as a pair
\begin{equation} \label{no3}
 (x, c_0 \to c_1 \to \ldots \to c_n)
\end{equation}
with $\pi(x) = c_0$, similar to \refdiag{nerve}, or as a pair
\begin{equation} \label{no4}
\xymatrix@R=1em{  & c_k \ar[dr] \ar@{.>}[d] \ar[dl] \\ (x, c_0 & \hat{c_k} & c_n )}
\end{equation}
with $\pi(x) = c_k$, as in \refdiag{nerveagain}.

\subsection{Simplicial groupoids} The definitions above make sense in any ambient category \ct{E} with pullbacks, giving rise to groupoid objects in \ct{E}, actions on maps $X \to {\rm ob}(\mathbb{G})$ in \ct{E}, and simplicial objects $N(\mathbb{G})$ and $N(X_\mathbb{G})$ in \ct{E}. If $C^\bullet$ is a cosimplicial object in \ct{E} and \ct{E} has the necessary colimits, one can next take the ``geometric realization''
\[ - \otimes_{\mathbb{\Delta}} C^\bullet: \ct{E}^{{\mathbb{\Delta}}^{op}} \to \ct{E}: Y \mapsto Y \otimes_{\mathbb{\Delta}} C^\bullet \]
to get back to \ct{E}.

We will only be interested in the case where \ct{E} is the category $\sSets$ of simplicial sets itself and $C^\bullet$ is the cosimplicial object of standard simplices, $C^n = \Delta[n]$. In this case, $\ct{E}^{\mathbb{\Delta}^{op}}$ is the category $\bisSets$ of bisimplicial sets and the geometric realization is (isomorphic to) the \emph{diagonal}, which we write as
\[ \delta^*: \bisSets \to \sSets. \]
It sends a bisimplicial set $Y = \{ Y_{p, q} \}$ to its diagonal $\delta^*(Y)$ given by $\delta^*(Y)_n = Y_{n, n}$. We shall have occasion to use the left adjoint to $\delta^*$,
\[ \delta_!: \sSets \to \bisSets \]
which is completely determined by its effect on representables,
\[ \delta_!(\Delta[n])_{p,q} = \{ (\alpha, \beta) \, : \, \alpha \in \Delta[n]_p, \beta \in \Delta[n]_q \}. \]
If $\mathbb{G}$ is a simplicial groupoid (a groupoid object in $\sSets$) its nerve is a bisimplicial set, and the diagonal thereof is denoted
\[ B\mathbb{G} := \delta^*(N \mathbb{G}) \]
and referred to as the \emph{classifying space} of $\mathbb{G}$ (using ``space'' as synonymous for simplicial set). In particular, if $\mathbb{G}$ acts on a map $\pi: X \to {\rm ob}(\mathbb{G})$ of simplicial sets, one obtains a map
\[ \pi: B(X_\mathbb{G}) \to B(\mathbb{G}) \]
of classifying spaces. Since both nerve and diagonal preserve pullbacks, one obtains for each object $c$ in $\mathbb{G}$ a pullback of simplicial sets
\diaglab{no5}{ X_c \ar[r] \ar[d] & BX_\mathbb{G} \ar[d]^\pi \\
1 \ar[r] & B\mathbb{G} }
where $1 = \Delta[0]$ is the one-point simplicial set and $X_c = \pi^{-1}(c) \subseteq X$.

\section{Fibration properties of simplicial groupoids}

We will now list some conditions under which simplicial groupoids and their actions give rise to Kan complexes and Kan fibrations. These statements are all rather elementary and will be proved in detail.

\begin{theo}{onsimplicialgroupoids}
Let $\mathbb{G}$ be a groupoid in $\sSets$, and (for part (ii)) let $\mathbb{G}$ act from the left on $X \to {\rm ob}(\mathbb{G})$.
\begin{enumerate}
\item[(i)] If ${\rm ob}(\mathbb{G})$ is a Kan complex and $s: {\rm ar}(\mathbb{G}) \to {\rm ob}(\mathbb{G})$ is a Kan fibration then $B\mathbb{G}$ is a Kan complex.
\item[(ii)] If $\pi: X \to {\rm ob}(\mathbb{G})$ is a Kan fibration then so is $B(X_\mathbb{G}) \to B\mathbb{G}$.
\item[(iii)] If $\mathbb{G}$ is transitive, i.e., if $(s, t): {\rm ar}(\mathbb{G}) \to {\rm ob}(\mathbb{G}) \times {\rm ob}(\mathbb{G})$ is surjective, then this map $(s, t): {\rm ar}(\mathbb{G}) \to {\rm ob}(\mathbb{G}) \times {\rm ob}(\mathbb{G})$ is a Kan fibration.
\item[(iv)] If $(s, t): {\rm ar}(\mathbb{G}) \to {\rm ob}(\mathbb{G}) \times {\rm ob}(\mathbb{G})$ is a Kan fibration, then for any object $c$ of $\mathbb{G}$, i.e., any vertex $c$ in ${\rm ob}(\mathbb{G})$, there is a natural weak equivalence $\mathbb{G}(c, c) \to \Omega(B \mathbb{G}, c)$.
\end{enumerate}
\end{theo}

\begin{rema}{ontheorem}
Part (iii) is what remains for groupoids of the classical fact that simplicial groups are Kan complexes \cite{gabrielzisman67,may67}. In part (iv), $\Omega(B \mathbb{G}, c)$ denotes the loop space of $B \mathbb{G}$ at the base point $c$. It can be calculated inside simplicial sets as the evident subspace of $B\mathbb{G}^{\Delta[1]}$ in case $B \mathbb{G}$ is Kan, cf.~(i). Part (iv) is well known for simplicial groups, of course. As to the proofs, we will see that (i) and (ii) can be verified by direct inspection, while (iv) is an immediate consequence of (ii). The proof of (iii) is slightly more involved.
\end{rema}

\begin{proof}[Proof of (i).] The case $n = 1$ being trivial (since $\Delta[0] \subseteq \Delta[1]$ is a retract), choose $n \gt 1$ and $0 \leq k \leq n$, and consider an extension problem of the form
\diag{ \Lambda^k[n] \ar[r]^\varphi \ar[d] & B\mathbb{G} \\
\Delta[n] \ar@{.>}[ur]_? }
By adjointness of $\delta_!$ and $\delta^*$, this is equivalent to an extension problem of the form
\diaglab{problem}{ \delta_! \Lambda^k[n] \ar[r]^\psi \ar[d] & N\mathbb{G} \\
\delta_! \Delta[n] \ar@{.>}[ur]_? }
in the category of bisimplicial sets. Also, by symmetry of the nerve of a groupoid (cf.~\refdiag{nerve} and \refdiag{nerveagain} in Section 2 above) it suffices to consider the case $k = 0$, and represent $n$-simplices in the nerve conveniently in the form
\diaglab{nsimplinnerve}{ & & c_0 \ar[dr]^{g_n} \ar[dl]^{g_2} \ar[dll]_{g_1}  \\
c_1 & c_2 & \ldots & c_n }
Thus, the data provided by $\psi$ in \refdiag{problem} are matching faces which are diagrams in $\mathbb{G}_{n-1}$ of the form
\diaglab{datagiven}{ & & d_ic_0 \ar[dr]^{d_ig_n} \ar@{.>}[dl]^{d_ig_i} \ar[dll]_{d_ig_1}  \\
d_ic_1 & \hat{d_ic_i} & \ldots & d_ic_n }
for all $i \gt 0$, and the problem is to extend this to a diagram \refdiag{nsimplinnerve}. (The notation in \refdiag{datagiven} is suggestive and should not lead to confusion: we are given these $d_ic_j$ and $d_ig_j$, but not the $c_j$ and $g_j$ themselves.) Since ${\rm ob}(\mathbb{G})$ is assumed to be Kan, we can first extend the $d_ic_0$, $i \gt 0$, to an object $c_0$ in $\mathbb{G}_n$,
\diag{ \Lambda^0[n] \ar[d] \ar[r]^f & {\rm ob}(\mathbb{G}) \\
\Delta[n] \ar@{.>}[ur]_{c_0}, }
where $f = {\{d_ic_0\}_{i \gt 0}}$ is the map sending for each $i \gt 0$ the face opposite the $i$th vertex to $d_i c_0$.
Next, for a fixed $j \gt 0$, we are given $d_ig_j :d_ic_0 \to d_ic_j$ for all $i \not=0, j$. Let $\Lambda^{0, j}[n] \to \Delta[n]$ be the union of all the faces except the ones opposite the 0th and $j$th vertex (i.e., the union of all the faces containing these two vertices). Then $\Lambda^{0, j}[n] \to \Delta[n]$ is an anodyne extension. Since $s: {\rm ar}(\mathbb{G}) \to {\rm ob}(\mathbb{G})$ is assumed to be a fibration, we can complete the diagram
\diag{ \Lambda^{0, j}[n] \ar[r]^{ \{d_ig_j\}_i} \ar[d] & {\rm ar}(\mathbb{G}) \ar[d]^s \\
\Delta[n] \ar[r]_{c_0} \ar@{.>}[ur]^{g_j} & {\rm ob}(\mathbb{G}) }
where $j$ is fixed and $i$ ranges over all $1, \ldots, \mathit{ \hat{\jmath}}, \ldots, n$, to get a map $g_j: \Delta[n] \to {\rm ar}(\mathbb{G})$ as indicated. Doing this for each $j \gt 0$ completes the family \refdiag{datagiven} into a diagram of the form \refdiag{nsimplinnerve}. This proves part (i) of the theorem.
\end{proof}

\begin{proof}[Proof of (ii).] We have to solve a lifting problem of the form
\diaglab{problagain}{ \Lambda^k[n] \ar[d] \ar[r] & \delta^*N(X_\mathbb{G}) \ar[d] \\
\Delta[n] \ar[r] & \delta^*N(\mathbb{G}) }
for each $n \geq 0$ and each $0 \leq k \leq n$. Again by symmetry, it suffices to prove this for $k = 0$, and we can conveniently represent simplices in the nerves of the form \refdiag{nerve} and (\ref{no3}) of Section 2. So the data provided by diagram \refdiag{problagain} are an $n$-simplex $\xymatrix@1{ c_0 \ar[r]^{g_0} & c_1 \ar[r] & \ldots \ar[r]^{g_n} & c_n}$ in $N(\mathbb{G})_n$ and for each $i \gt 0$ an element $x_i \in X_{n-1}$ with $\pi(x_i) = d_i c_0$, agreeing on overlapping faces. In other words, diagram \refdiag{problagain} provides us with a commutative square
\diaglab{otherdiagr}{ \Lambda^0[n] \ar[r]^{ \{ x_i \} } \ar[d] & X \ar[d]^{\pi} \\
\Delta[n] \ar@{.>}[ur]^x \ar[r]_{c_0} & {\rm ob}(\mathbb{G}) }
Since $\pi$ is assumed to be a fibration, there exists a diagonal $x: \Delta[n] \to X$ in \refdiag{otherdiagr}, which defines an $n$-simplex $(x, c_0 \to \ldots \to c_n)$ in $N(X_\mathbb{G})$ and hence a diagonal in \refdiag{problagain}.
\end{proof}

Before proving Part (iii), we observe the following elementary properties of Kan fibrations.

\begin{lemm}{descent} {\rm (``Descent'')} Consider a pullback diagram of simplicial sets
\diag{ Y' \ar[d]_{p'} \ar@{->>}[r] & Y \ar[d]^p \\
X' \ar@{->>}[r]_f & X }
in which $f$ is surjective. If $p'$ is a Kan fibration, then so is $p$.
\end{lemm}
\begin{proof}
Immediate from the definitions.
\end{proof}

\begin{lemm}{quotients} {\rm (``Quotients'')} In a diagram
\diag{ Z \ar@{->>}[rr]^p \ar[dr]_g & & Y \ar[dl]^f \\
& X, }
if $g = f \circ p$ is a Kan fibration and $p$ is a surjective Kan fibration, then $f$ is a Kan fibration.
\end{lemm}
\begin{proof} (See also \cite[Proposition 4.1]{bergmoerdijk15}.) Consider a lifting problem as on the left
\begin{displaymath}
\begin{array}{ccc}
\xymatrix{ \Lambda^k[n] \ar[r]^b \ar[d]_i & Y \ar[d]^f \\
\Delta[n] \ar[r]_a \ar@{.>}[ur]^? & X } & &
\xymatrix{ \Delta[0] \ar[r] \ar[d] & Z \ar[d]^p \\
\Lambda^k[n] \ar[r]_b \ar@{.>}[ur]^c & Y }
\end{array}
\end{displaymath}
Let $k: \Delta[0] \to \Lambda^k[n]$ be the $k$-th vertex. Then by assumption, $b \circ k$ lifts to $Z$ and we can next fill the square on the right. Since $g$ is a Kan fibration and $gc = ai$, there exists a map $h: \Delta[n] \to Z$ with $gh = a$ and $hi = c$. Then $ph: \Delta[n] \to Y$ is a the required lift in the square on the left above.
\end{proof}

\begin{lemm}{lemmaonsimplgrps}
Let $H$ be a simplicial group acting freely on a simplicial set $E$. Then $E \to E/H$ is a Kan fibration.
\end{lemm}
\begin{proof}
Let us write $X = E/H$ and $q: E \to X$ for the quotient map. Then the lemma simply states that the principal $H$-bundle $E \to X$ is a Kan fibration. This is well known, but here is an elementary proof. Consider the diagram
\diag{ H \ar[d] & H \times E \ar[l] \ar[r]^{\theta} \ar[d]_{\pi_2} & E \times_X E \ar[d]^{\pi_2} \ar[r] & E \ar[d]^q \\
\Delta[0] & E \ar[l] \ar[r]_{=} & E \ar[r] & X }
where $\theta$ is the isomorphism $\theta(h, e) = (h \cdot e, e)$. Since simplicial groups are Kan, $H \to \Delta[0]$ is a fibration, and hence so are its pullback $H \times E \to E$ and the isomorphic map $\pi_2: E \times_X E \to E$. By \reflemm{descent}, $q: E \to X$ is a Kan fibration.
\end{proof}

\begin{proof}[Proof of (iii).] Fix an object $c$ in $\mathbb{G}$, and write $H$ for the simplicial group $\mathbb{G}(c,c)$. Consider the pullback
\diag{ E \ar[r] \ar[d]_s & {\rm ar}(\mathbb{G}) \ar@{->>}[d]^{(s,t)} \\
{\rm ob}(\mathbb{G}) \ar[r]_(.35){(1, c)} & {\rm ob}(\mathbb{G}) \times {\rm ob}(\mathbb{G}). }
So $E = \mathbb{G}(-, c)$ is the simplicial set of arrows into $c$. The group $H$ acts on $E$ by composition, and this defines a principal $H$-bundle by assumption of surjectivity of $(s, t)$. In particular, $s: E \to {\rm ob}(\mathbb{G})$ is a Kan fibration by \reflemm{lemmaonsimplgrps}. The group $H$ also acts freely on the product $E \times E$ by the diagonal action, so by the same lemma, $E \times E \to (E \times E)/H$ is also a Kan fibration. But $(E \times E)/H$ is isomorphic to ${\rm ar}(\mathbb{G})$ by an isomorphism that fits into
\diag{ (E \times E)/H \ar[r]^{\cong} \ar[dr] & {\rm ar}(\mathbb{G}) \ar[d]^{(s, t)} \\
E \times E \ar@{->>}[u] \ar[r]_(.4){s \times s} & {\rm ob}(\mathbb{G}) \times {\rm ob}(\mathbb{G}) }
The diagonal is Kan fibration by \reflemm{quotients}, and hence so is the map on the right.
\end{proof}

\begin{proof}[Proof of (iv).] Again fix an object $c$ in $\mathbb{G}$, and now consider the pullback
\diag{ D \ar[r] \ar[d]_t & {\rm ar}(\mathbb{G}) \ar[d]^{(s, t)} \\
{\rm ob}(\mathbb{G}) \ar[r]_(.35){(s, 1)} & {\rm ob}(\mathbb{G}) \times {\rm ob}(\mathbb{G}) }
So $D = \mathbb{G}(c, -)$ is the simplicial set of arrows out of $c$. Then $t: D \to {\rm ob}(\mathbb{G})$ is a Kan fibration since $(s, t)$ is assumed to be. Composition of $\mathbb{G}$ defines a left action by $\mathbb{G}$ on $D$, so by Part (ii) of the theorem we obtain a Kan fibration $B(D_\mathbb{G}) \to B(\mathbb{G})$ with fibre $D_c = \mathbb{G}(c, c)$ (cf.~the pullback square \refdiag{no5} in Section 2). But $D_\mathbb{G}$ is the simplicial groupoid $c/\mathbb{G}$ which has an initial object. So $B(D_\mathbb{G})$ is contractible, and the fibre of $B(D_\mathbb{G}) \to B(\mathbb{G})$ is the loop space.
\end{proof}

\section{Universal Kan fibrations and univalence}

Let us call a simplicial set \emph{small} if it is countable, and a map \emph{small} if it has countable fibres. (Countability is irrelevant here, in the sense that it could be replaced by any other bound by an infinite regular cardinal.)

Note that if $f: Y \to X$ is small, then Quillen's small object argument \cite{gabrielzisman67} gives a factorisation of $f$ into an anodyne extension $Y \to Y'$ and a \emph{small} fibration $Y' \to X$, More generally, any small $f$ fits into a square
\diag{ Y \ar@{ >->}[r]^{\sim} \ar[d]_f & Y' \ar[d]^{f'} \\
X \ar@{ >->}[r]^{\sim} & X' }
with horizontal anodyne extensions and a small fibration $Y' \to X'$ into a Kan complex.

\begin{defi}{universal} A small fibration $\pi: E \to U$ is called \emph{universal} if every small fibration $Y \to X$ fits into a homotopy pullback
\diag{ Y \ar[r] \ar[d] & E \ar[d] \\
X \ar[r] & U.}
(That is, the map from $Y$ to the pullback is a (weak) homotopy equivalence.)
\end{defi}

\begin{rema}{typetheory} In the literature related to type theory, one also considers a stronger notion of universality, where every small fibration is \emph{isomorphic} (rather than homotopy equivalent) to a pullback of the universal $E \to U$. Let us call an $E \to U$ with this property \emph{strictly universal}. Obviously every strictly universal fibration is universal. But strictly universal maps are not very well behaved from the point of view of homotopy theory (cf. \refprop{propofuniversal}.(iv) below). Moreover, for several type theoretic applications of strict universality, the weaker notion defined above suffices. (This applies, for example, to the construction of models of {\bf CZF}, or to the fact that univalence implies function extensionality \cite{bergmoerdijk16}.)
\end{rema}

We now state some elementary properties of universal fibrations. We omit the proofs, which are all obvious.

\begin{prop}{propofuniversal}
\begin{enumerate}
\item[(i)] If
\diag{ E \ar[d]_p \ar[r] & E' \ar[d]^{p'} \\
U \ar[r] & U' }
is a homotopy pullback where $p$ and $p'$ are small fibrations, then $p'$ is universal whenever $p$ is.
\item[(ii)] In particular, any small universal fibration $E \to U$ can be ``completed'' into another one $E' \to U'$ with a Kan complex $U'$ as base, as in
\diag{ E \ar@{ >->}[r]^{\sim} \ar[d] & E' \ar[d] \\
U \ar@{ >->}[r]^{\sim} & U'. }
\item[(iii)] In a homotopy pullback square as in (i), if $\xymatrix@1{U \ar[r]^{\sim} & U'}$ is a weak equivalence and $U$ is Kan then $p$ is universal whenever $p'$ is.
\item[(iv)] In particular, as a property of small fibrations with a Kan complex as a base, being universal is invariant under (weak) homotopy equivalence.
\end{enumerate}
\end{prop}

Let $E \to U$ be any fibration. We write ${\rm End}(E) \to U \times U$ for the fibration constructed as the exponential $\pi_2^*(E)^{\pi_1^*(E)} \to U \times U$ in $\sSets/ U \times U$. Thus, the fibre over a pair $(x, y)$ of vertices of $U$ is $E_y^{E_x} = {\rm Hom}(E_x, E_y)$. This fibration contains a subfibration ${\rm Eq}(E) \to U \times U$ of weak equivalences between fibres. More explicitly, an $n$-simplex of ${\rm Eq}(E)$ over $(x, y): \Delta[n] \to U \times U $ is a weak equivalence $x^* E \to y^* E$ over $\Delta[n]$. (The map ${\rm Eq}(E) \to U \times U$ is again a fibration, since a map $x^* E \to y^* E$ as above is a weak equivalence iff it is one over one or all vertices of $\Delta[n]$.)

Next, consider the ``constant path'' inclusion $U \to U^{\Delta[1]}$ and its pullback along $U^{\Delta[1]} \times_U E \to U^{\Delta[1]}$. A diagonal filling in the diagram below
\diag{ U \ar@{ >->}[d] & E \ar@{ >->}[d] \ar[l] \ar[r]^1 & E \ar[d] \\
U^{\Delta[1]} & U^{\Delta[1]} \times_U E \ar[l]  \ar[r] \ar@{.>}[ur]^\nabla & U }
gives a ``connection'' $\nabla: U^{\Delta[1]} \times_U E \to E$ over $U$, or $U^{\Delta[1]} \to {\rm End}(E)$ over $U \times U$, which is easily seen to factor through ${\rm Eq}(E) \subseteq {\rm End}(E)$. The fibration $E \to U$ is said to be \emph{univalent} \cite{kapulkinetal14} if this map $U^{\Delta[1]} \to {\rm Eq}(E)$ is a weak equivalence.

\begin{rema}{conncompinthebase} If $x, y \in U$ are two vertices in the same component of $U$ then $\nabla$ provides a (zigzag of) weak equivalence(s) $E_x \to E_y$. Suppose, conversely, that there is no weak equivalence $E_x \to E_y$ if $x$ and $y$ belong to different components of $U$. Then by the long exact sequence of a fibration, univalence is equivalent to the statement that $\nabla$ induces for each base point $x_0 \in U$ a weak equivalence
\[ \nabla: \Omega(U, x_0) \to {\rm Eq}(E_{x_0}), \]
where $\Omega(U, x_0)$ is the loop space of $U$ at $x_0$ (constructed as the homotopy fibre of $U^{\Delta[1]} \to U \times U$, or simply as the fibre if $U$ is Kan) and ${\rm Eq}(E_{x_0})$ is the simplicial set of self-equivalences of $E_{x_0}$. Indeed, in order to compare the (homotopy) fibres of
\diag{ U^{\Delta[1]} \ar[dr] \ar[rr] & & {\rm Eq}(E) \ar[dl] \\
& U \times U }
we can restrict ourselves to diagonal base points, because the (homotopy) fibre of either map over $(x, y)$ is non-empty iff $x$ and $y$ belong to the same connected component.
\end{rema}

\begin{rema}{univinvariant} Univalence is invariant under (weak) homotopy equivalence, in the sense that if
\diag{ E \ar[d]_p \ar[r]^{\sim} & E' \ar[d]^{p'} \\
U \ar[r]^{\sim} & U' }
is a weak equivalence between two fibrations $p$ and $p'$, then $p$ is univalent iff $p'$ is. Vice versa, if $E \to U$ and $E' \to U'$ are two univalent universal fibrations then they fit into such a square. In this sense, ``the'' universal univalent fibration is unique up to homotopy.
\end{rema}

\section{Univalent completion}

In this section, we will show that any fibration can be embedded into a univalent one, in the following sense.

\begin{theo}{univcompl} Let $E \to U$ be a fibration. Then there exists a homotopy pullback square
\diag{ E \ar[r] \ar[d]_p & E' \ar[d]^{p'} \\
U \ar[r] & U' }
where $p'$ is univalent and $U \to U'$ is mono. Moreover, if $p$ is small then so is $p'$, and if $p$ is universal then so is $p'$.
\end{theo}

\begin{rema}{easy} It will be obvious from the construction that $p'$ is small whenever $p$ is. Moreover, universality of $p'$ follows from that of $p$ by \refprop{propofuniversal}. Note that, in contrast, any universal fibration $E \to U$ can be trivially embedded into a universal fibration which is not univalent, such as $E + Y \to U + X$ where $Y \to X$ is an arbitrary fibration satisfying the size restrictions.
\end{rema}

\begin{proof}[Proof of the Theorem.] Fix $p: E \to U$, and choose a minimal fibration inside $E$ \cite{gabrielzisman67,may67}:
\diag{ M \ar[rr]^i \ar[dr] & & E \ar[dl] \\
& U. }
Thus, $M$ is a fibrewise deformation retract of $E$. Moreover, by minimality, any weak equivalence $M_x \to M_y$ between fibres of $M$ is an isomorphism. Thus, we obtain maps
\diag{ {\rm Iso}(M) = {\rm Eq}(M) \ar[r]^(.62){\sim} & {\rm Eq}(E),}
and $M \to U$ is homotopy equivalent to $E \to U$ hence universal whenever $E \to U$ is. Thus, to prove the theorem, we might as well assume that $E \to U$ is a \emph{minimal} fibration, as we will now do.

Let $\mathbb{G}$ be the simplicial groupoid with ${\rm ob}(\mathbb{G}) = U$ and ${\rm ar}(\mathbb{G}) = {\rm Iso}(E)$. In other words, an $n$-simplex in ${\rm ar}(\mathbb{G})$ is a triple $(x, y, \alpha)$ where $x, y: \Delta[n] \to U$ and $\alpha: x^* E \to y^* E$ is an isomorphism over $\Delta[n]$. Then $\mathbb{G}$ acts on $E \to U$ in the obvious way, so we obtain a pullback square of simplicial sets
\diag{ E \ar@{ >->}[r]^j \ar[d]_p & B(E_\mathbb{G}) \ar[d]^{p'} \\
U \ar@{ >->}[r]_i & B\mathbb{G}. }
The map $p'$ on the right is a Kan fibration by \reftheo{onsimplicialgroupoids}.(ii). Moreover, since ${\rm ar}(\mathbb{G}) \to {\rm ob}(\mathbb{G}) \times {\rm ob}(\mathbb{G})$ is the Kan fibration ${\rm Iso}(E) = {\rm Eq}(E) \to U \times U$, there is weak equivalence $\mathbb{G}(x_0, x_0) \to \Omega(B\mathbb{G}, x_0)$ for any vertex $x_0$ in $U$. But $\mathbb{G}(x_0, x_0) = {\rm Iso}(E_{x_0}) = {\rm Eq}(E_{x_0})$, so this proves that $p'$ is univalent provided the condition on connected components is satisfied (cf.~\refrema{conncompinthebase} above). We conclude the proof by checking this condition: The map $i: U \to B\mathbb{G}$ is an isomorphism on vertices, and the fibre of $p'$ over a vertex $i(x)$ is $E_x$. If $x$ and $y$ are in the same connected component of $B\mathbb{G}$ then they are related by a weak equivalence provided by a ``connection'' $\nabla$ for $p'$. And conversely, if $E_x$ and $E_y$ are related by a weak equivalence, then by construction there is an arrow in $\mathbb{G}$ from $x$ to $y$, hence $x$ and $y$ are in the same connected component of $B\mathbb{G}$. This completes the proof.
\end{proof}

\begin{rema}{notminimal}
In connection with univalence in other model categories, it is perhaps of interest to remark that it is possible to avoid the use of minimal fibrations in the proof of \reftheo{univcompl}. Instead, one can use the following version of the ``group completion theorem'' from \cite{moerdijk89}, concerning a category object $\mathbb{C}$ in $\sSets$ (rather than a groupoid object $\mathbb{G}$ considered before) and an action by $\mathbb{C}$ on a map $X \to {\rm Ob}(\mathbb{C})$.
\begin{theo}{catcomplth}
If $\mathbb{C}$ acts on $X$ by weak equivalences, then the pullback square
\diag{ X_c \ar[r] \ar[d] & BX_{\mathbb{C}} \ar[d] \\
1 \ar[r] & B\mathbb{C} }
is a homotopy pullback.
\end{theo}
This is stated in \cite[Theorem 2.1]{moerdijk89} for a category object $\mathbb{C}$ with ${\rm ob}(\mathbb{C})$ discrete, but this plays no r\^ole in the proof given there. To say that $\mathbb{C}$ acts by weak equivalences simply means that for each arrow $\alpha: c \to d$ in $\mathbb{C}_0$, the induced map $X_c \to X_d$ is a weak equivalence. \reftheo{univcompl} can now be proved by applying \reftheo{catcomplth} to the category $\mathbb{C}$ defined by ${\rm ob}(\mathbb{C}) = U$ and ${\rm ar}(\mathbb{C}) = {\rm Eq}(E)$, the space of equivalences between fibres of $E \to U$.
\end{rema}

\section{Acknowledgements}

The results of this paper were presented by the second named author at workshops in Sheffield and at the Mittag-Leffler Institute in Stockholm in March and June 2015, respectively. We are especially grateful to the Institute: it was there that this paper acquired its final form.

\bibliographystyle{plain} \bibliography{hSetoids}

\end{document}